\documentclass[11pt]{amsart}
\usepackage{fullpage,cite,hyperref,enumitem}
\usepackage{amssymb,amsmath,amsthm,mathtools,extpfeil}
\usepackage[all,cmtip]{xy}
\usepackage{rotating,graphicx}
\usepackage{tikz}
\newtheorem{thm}{Theorem}[section]
\newtheorem*{thm*}{Theorem}

\newtheorem{lem}[thm]{Lemma}
\newtheorem{prop}[thm]{Proposition}
\newtheorem{cor}[thm]{Corollary}
\newtheorem{conj}{Conjecture}

\theoremstyle{definition}
\newtheorem*{defn}{Definition}

\numberwithin{equation}{section}

\DeclareMathOperator{\poly}{poly}

\DeclareMathOperator{\id}{id}

\DeclareMathOperator{\Lip}{Lip}

\DeclareMathOperator{\vol}{vol}

\newcommand{\ph}{\varphi}
\newcommand{\epsi}{\varepsilon}

\begin{document}
\title{Quantitative nullhomotopy and rational homotopy type}
\author{Gregory R. Chambers}
\address{Department of Mathematics, University of Chicago, Chicago, Illinois, USA}
\email[G.~R.~Chambers]{chambers@math.uchicago.edu}
%\author{Dominic Dotterrer}
%\address{Department of Mathematics, University of Chicago, Chicago, Illinois, USA}
\author{Fedor Manin}
\address{Department of Mathematics, University of Toronto, Toronto, Ontario, Canada}
\email[F.~Manin]{manin@math.toronto.edu}
\author{Shmuel Weinberger}
\address{Department of Mathematics, University of Chicago, Chicago, Illinois, USA}
\email[S.~Weinberger]{shmuel@math.uchicago.edu}

\begin{abstract}
  In \cite{GrOrang}, Gromov asks the following question: given a nullhomotopic
  map $f:S^m \to S^n$ of Lipschitz constant $L$, how does the Lipschitz constant
  of an optimal nullhomotopy of $f$ depend on $L$, $m$, and $n$?  We establish
  that for fixed $m$ and $n$, the answer is at worst quadratic in $L$.  More
  precisely, we construct a nullhomotopy whose \emph{thickness} (Lipschitz
  constant in the space variable) is $C(m,n)(L+1)$ and whose \emph{width}
  (Lipschitz constant in the time variable) is $C(m,n)(L+1)^2$.

  More generally, we prove a similar result for maps $f:X \to Y$ for any compact
  Riemannian manifold $X$ and $Y$ a compact simply connected Riemannian manifold
  in a class which includes complex projective spaces, Grassmannians, and all
  other simply connected homogeneous spaces.  Moreover, for all simply connected
  $Y$, asymptotic restrictions on the size of nullhomotopies are shown to be
  determined by rational homotopy type.
\end{abstract}
\maketitle
\setcounter{tocdepth}{1}
\tableofcontents
%%%%%%%%%%%%%%%%%%%%%%%%%%%%%%%%%%%%%%%%%%%%%%%%%%%%%%%%%%%%%%%%%%%%%%%%%%%%%%%
\section{Introduction}
%%%%%%%%%%%%%%%%%%%%%%%%%%%%%%%%%%%%%%%%%%%%%%%%%%%%%%%%%%%%%%%%%%%%%%%%%%%%%%%

Rational homotopy theory, as introduced by Quillen and Sullivan, is one of the
great successes of twentieth-century algebraic topology.  It allows one to turn
any simply connected space, which may be given as a Postnikov tower or a cell
complex, in a rather simple algorithmic way, into one of several, ultimately
equivalent, algebraic structures.  Moreover, as long as one is willing to ignore
torsion, this conversion preserves all homotopic information: it is an
equivalence of (rational homotopy) categories.

As in other such cases, we often understand very little about the geometry of
the maps that rational homotopy theory tells us must exist.  Nevertheless, quite
a bit of geometric information may be squeezed out of this algebraic story.

Perhaps the earliest theorem of quantitative algebraic topology is the following,
stated by Gromov in \cite{GrHED}:
\begin{thm} \label{thm:XtoY}
  Let $X$ and $Y$ be compact simply connected Riemannian manifolds.  Then
  $$\#\{[f] \in [X:Y]: \Lip f \leq L\}=O(L^{\alpha}),$$
  where $\alpha$ depends only on the rational homotopy type of $X$ and $Y$.
\end{thm}
Na\"\i vely, rational homotopy type should play a role here because torsion
homotopy groups can only affect the number of maps by a finite multiplicative
constant.  But in fact, the proof of this theorem relies heavily on the Sullivan
model of rational homotopy theory and its realization via differential forms.

Later on, Gromov \cite{GrQHT} conjectured that an analogous result should hold
for homotopies between maps.  To state the conjecture, we first introduce some
terminology.  Suppose that $X$ and $Y$ are two metric simplicial complexes, and
that $f, g:X \rightarrow Y$ are two homotopic maps.  We say a homotopy
$H:X \times [0,1] \rightarrow Y$ from $f$ to $g$ has \emph{thickness} $A$ and
\emph{width} $B$ if $d(H(x,t),H(y,t)) \leq Ad(x,y)$ for all $x$, $y$, and $t$
and $d(H(x,t),H(x,s)) \leq B\lvert t-s \rvert$ for all $x$, $t$, and $s$.

Gromov's original conjecture concerned only the thickness of homotopies:
\begin{conj} \label{conjGr}
  Let $X$ and $Y$ be compact Riemannian manifolds (or some other ``reasonable''
  class of compact metric spaces) with $Y$ simply connected.  If $f,g:X \to Y$
  are homotopic maps with Lipschitz constant $\leq L$, then there is a homotopy
  between them of thickness $O(L^p)$, for some $p$ depending only on the
  rational homotopy type of $Y$.  Perhaps $p$ can always be taken to be 1.
\end{conj}
In \cite{FWPNAS}, Ferry and Weinberger suggest a related problem: can the
Lipschitz constant of a homotopy, seen as a map $X \times [0,1] \to Y$, be
bounded linearly in terms of $\Lip f$ and $\Lip g$?  As we show in \cite{cob},
this is not the case in general.  However, one may hope for a polynomial result.
In light of the result in this paper, it may be worthwhile to consider thickness
and width separately.  A compelling if somewhat optimistic conjecture is as
follows:
\begin{conj} \label{conjAll}
  In the setting of Conjecture \ref{conjGr}, if $f,g: X \to Y$ are homotopic
  maps with Lipschitz constant $\leq L$, then there is a homotopy between them
  of thickness $O(L)$ and width $O(L^p)$, where $p$ depends only on the rational
  homotopy type of $Y$.
\end{conj}
In the case of nullhomotopic maps, we make a stronger conjecture which refers
explicitly to the rational homotopy type of $Y$:
\begin{conj} \label{conjNull}
  If $f: X \to Y$ is nullhomotopic with Lipschitz constant $\leq L$, then it
  has a nullhomotopy of thickness $O(L)$ and width $O(L^q)$, where $q$ is the
  minimal depth of a filtration $0=V_0 \subset V_1 \subset \cdots \subset V_q$
  of the indecomposables in dimensions $\leq n$ of the Sullivan minimal model of
  $Y$ with the property that $dV_i \subseteq \mathbb{Q}\langle V_{i-1} \rangle$.
\end{conj}
The cases $q=0$ and $q=1$ of this conjecture are proved in \cite{FWPNAS} and
\cite{cob}, respectively.  Indeed, those results hold for homotopies and not
only nullhomotopies.  In this paper we prove the case $q=2$:
\begin{thm} \label{intro:main}
  Let $X$ and $Y$ be finite simplicial complexes, with $X$ $n$-dimensional.  If
  $Y$ is simply connected and the indecomposables in dimensions $\leq n$ of its
  Sullivan minimal model split as $V_1 \oplus V_2$ with $dV_1=0$ and $dV_2
  \subset \mathbb{Q}\langle V_1 \rangle$, then there is a constant $C(X,Y)$ such
  that nullhomotopic $L$-Lipschitz maps from $X$ to $Y$ admit nullhomotopies
  of thickness $C(L+1)$ and width $C(L+1)^2$.
\end{thm}
The class of target spaces covered by this theorem includes, most notably, all
simply connected homogeneous spaces, including spheres.  As a corollary, when the
domain is a suspension, this allows us to find short homotopies, not just
nullhomotopies:
\begin{cor}
  In the setting of Theorem \ref{intro:main}, if in addition $X$ has the
  homotopy type of a suspension, there is a constant $C^\prime(X,Y)$ such that
  any two homotopic $L$-Lipschitz maps $f,g:X \to Y$ have a homotopy of
  thickness $C^\prime(L+1)$ and width $C^\prime(L+1)^2$.
\end{cor}
In particular, this gives a result for maps between spheres.  In
\cite[\S2]{GrOrang}, Gromov asks the following related question: given an
$L$-Lipschitz nullhomotopic map $f:S^m \to S^n$, how can the Lipschitz constant
of a nullhomotopy be bounded as a function of $m$, $n$ and $L$?  In this paper
we get a bound of the form $C(m,n)L^2$.  Getting an explicit estimate for
$C(m,n)$ is a topic for future work which is likely to require some geometric
understanding of homotopy groups of spheres.

As further confirmation that rational homotopy type plays a role, we prove the
following theorem, which holds for all simply connected targets:
\begin{thm} \label{intro:Qinv}
  Suppose $X$, $Y$, and $Z$ are finite complexes, with $X$ $n$-dimensional, and
  suppose that $Y$ and $Z$ are simply connected and rationally homotopy
  equivalent.  Then nullhomotopic maps $X \to Y$ and $X \to Z$ admit
  nullhomotopies \emph{of the same shapes}.
\end{thm}
We make this more precise below, but for example, if one has certain bounds on
thickness and width for nullhomotopies of maps to $Y$, then the same asymptotic
bounds hold for maps to $Z$.  It would be surprising if this didn't hold for all
homotopies rather than just nullhomotopies, but our proof does not generalize.

This paper is in large part a sequel to \cite{cob}.  While we explicitly restate
all the definitions and results we are using from that paper, the reader who has
absorbed its main techniques will have an easier time with the more complicated
cases covered here.

\subsection{The role of rational homotopy}
The main geometric tool that we use to construct quantitative nullhomotopies can
be seen in a simple example covered by the results in \cite{cob}: maps $S^n \to
S^n$.  To nullhomotope such a map, it is enough to cancel point preimages with
opposite local degree; this is an idea that goes back to Brouwer.  Tracing these
point preimages through the nullhomotopy gives an embedded 1-manifold in
$S^n \times [0,1]$.  In order to make the nullhomotopy quantitative, we break up
$S^n \times [0,1]$ into a grid and make sure that each cube in the grid doesn't
``see'' too much of this 1-manifold.

Poincar\'e duality turns this story about 0- and 1-submanifolds into one about
bounded $n$- and $(n-1)$-dimensional obstruction cochains which generalizes to a
result for maps $X \to K(\mathbb{Z},n)$ for any finite simplicial complex $X$.
In order to generalize this to a larger class of target spaces, one may try to
iterate this process over the stages of a Postnikov system, using the
obstruction theory for principal fibrations.

As pointed out by Gromov, tracing constants through the way such a lift is built
traditionally gives a Lipschitz constant which is a tower of exponentials.  Thus
to get a reasonable quantitative estimate, we need to once again do the lifting
in a local way.  Unfortunately, there is no guarantee that the nullhomotopy we
came up with in the previous stage is anywhere close to something that lifts.
For example, suppose that we are trying to nullhomotope an $L$-Lipschitz map
$f:S^3 \to S^2$ and we have come up with a nullhomotopy $F:S^3 \times [0,1] \to
\mathbb{C}\mathbf{P}^2$ which is cellular on a subdivision of a cell structure
$S^3 \times [0,1]$ at scale $1/L$.  We would like to retract this $F$ to a
nullhomotopy $S^3 \times [0,1] \to S^2$.  But a priori, the map $F|_{t=1/2}:S^3
\to S^2$ need not be nullhomotopic inside of $S^2$; by assumption, it is only
nullhomotopic as a map to $\mathbb{C}\mathbf{P}^2$.  Indeed, unless our
construction of $F$ was particularly clever, it may have Hopf invariant on the
order of $L^4$.  This means it cannot be made nullhomotopic even after a
homotopy in $\mathbb{C}\mathbf{P}^2$ if that homotopy is to be kept uniformly
bounded.

Our (partial) solution to this problem is to turn to algebra.  Let $Y$ be a
compact metric simplicial complex.  Its algebra of PL forms $A^*(Y)$ has a
Sullivan minimal model: a differential graded algebra (DGA) $\mathcal{M}^*(Y)$
which efficiently encodes the rational homotopy theory of the space and which is
realized by a map $R_Y:\mathcal{M}^*(Y) \to A^*(Y)$ inducing an isomorphism on
cohomology.  In particular, given a nullhomotopic map $f:X \to Y$, the map
$f^* \circ R_Y:\mathcal{M}^*(Y) \to A^*(X)$ is algebraically nullhomotopic, that
is, there is a homomorphism of DGAs
$$h:\mathcal{M}^*(Y) \to A^*(X) \otimes \mathbb{Q}\langle t,dt \rangle$$
with $h|_{t=0,dt=0}=f^* \circ R_Y$ and $h|_{t=1,dt=0}=0$.  In particular, this
homomorphism must commute with the differential, given on the codomain by a
graded Leibniz rule.

To construct such a homomorphism, one needs to antidifferentiate certain forms.
For example, if $X=S^3$ and $Y=S^2$, then $\mathcal{M}^*(Y)$ is given by
$\langle x_2,y_3 \mid dx=0,dy=x^2 \rangle$ and $R_Y$ takes $x$ to a volume form
$\omega$ on $S^2$ and $y$ to 0.  Then given a nullhomotopic map $f:S^3 \to S^2$
which is simplicial on a triangulation of $S^3$ at scale $1/L$, a nullhomotopy
$\tilde h:\mathcal{M}^*(S^2) \to A^*(S^3) \otimes \mathbb{Q}\langle t,dt \rangle$
can be given by
\begin{align*}
  x &\mapsto f^*\omega \otimes (1-t)^2-\alpha \otimes 2(1-t)dt \\
  y &\mapsto \eta \otimes 4(1-t)^3dt,
\end{align*}
where $\alpha$ and $\eta$ are defined so as to satisfy $d\alpha=f^*\omega$ and
$d\eta=f^*\omega \wedge \alpha$.  Here, $f^*\omega \wedge \alpha$ is an exact form
since the Hopf invariant of $f$ is zero, using J.H.C.~Whitehead's definition of
the Hopf invariant via integrals.  A filling inequality allows us to choose an
$\alpha$ of $\infty$-norm $O(L)$ and---since $\alpha$ and $f^*\omega$ each
contribute a factor of $L$---an $\eta$ of $\infty$-norm $O(L^2)$.

(Note that the various polynomials in $t$ can be replaced by other polynomials,
or even, once we leave purely algebraic territory, by any functions of $t$ which
satisfy the differential equations induced by the requirements on $\tilde h$.
The choice of these functions affects our final estimates only up to a constant.)

By ``evaluating'' $t$ and $dt$ we can turn this nullhomotopy into a map
$h:\mathcal{M}^*(S^2) \to A^*(S^3 \times I)$; this notation tacitly assumes a
simplicial structure on $S^3 \times I$ whose choice may depend on $L$.  If we
choose a fine enough subdivision of the interval, into $O(L^2)$ pieces, so that
simplices are very skinny in the time direction, then $dt$ is small enough that
the integrals over simplices of $h(x)$ and $h(y)$ are bounded uniformly,
independent of $L$.

Consider now the previously constructed nullhomotopy $F:S^3 \times [0,1] \to
\mathbb{C}\mathbf{P}^2$, and let $\xi$ be a differential form representing the
fundamental class of $\mathbb{C}\mathbf{P}^2$.  If $F^*\xi$ is a bounded distance
from $h(x)$, then Hopf invariants on boundaries of 4-cells can be determined by
integrating a form a bounded distance from $h(y)$.  Combined with ideas from
\cite{FWPNAS} and \cite{cob}, this allows us to kill these Hopf invariants by
modifying the map in a bounded way.

Unfortunately, if we try to continue this process to a third level and beyond,
the ``errors'' are no longer uniformly bounded.  This is related to the
well-known fact that $(L+1)^2 \neq L^2+1$.  This is why what seems to be the
second step of an induction does not actually generalize to a proof
Conjecture \ref{conjNull} for any $q \geq 3$.  At this time we have to be
content with Theorem \ref{intro:main}.

Another issue with potentially extending this method is that the property of
$\tilde h$ that one can cancel out large antiderivatives by making $dt$ small is
special: one can only construct such a nullhomotopy when $Y$ \emph{has positive
weights}, that is, essentially when $\mathcal{M}^*(Y)$ has lots of
automorphisms.  This property is discussed in \cite{BMSS} and examples of spaces
which do not have it are given in \cite{MT} and \cite{Am}.  For more general
spaces, such nullhomotopies may necessarily have large terms which are not
multiples of $dt$.  Thus, if one is to find a counterexample to Conjecture
\ref{conjNull} in which a nullhomotopy must necessarily have nonlinear
thickness, spaces which do not admit positive weights seem to be a natural place
to look.

\subsection{Optimality}
One may ask to what extent our results are sharp.  We produce two main examples
to this effect.  First, we give a sequence of examples, also mentioned in
\cite{cob}, which demonstrate that a linear bound does not always hold, and in
some cases the quadratic bound on width is the best we can do.  More generally,
this family of examples demonstrates that, at least in some cases, the
conjectured upper bound of Conjecture \ref{conjNull} is also a lower bound.
Secondly, we construct an example which shows that the statement of Theorem
\ref{intro:main} does not hold if we replace nullhomotopies by homotopies: the
exponents in Conjectures \ref{conjAll} and \ref{conjNull} are necessarily
different.

Nevertheless, many open questions remain even in the restricted domain of
Theorem \ref*{intro:main}.  There is some indication that for maps $S^3 \to S^2$
our quadratic bound on widths of nullhomotopies is not sharp, but rather is an
artifact of the algebraic method: it is possible to construct nullhomotopies with
subquadratic, perhaps even linear Lipschitz constant.

\subsection{Outline of the paper}
In section 2, we introduce and summarize some geometric results and terminology.
This is followed in the third section by a proof of Theorem \ref{intro:main} in
the special case of maps $S^3 \to S^2$.  Section 4 repeats this for a more
general, but still restricted situation.  In section 5 we prove that the
asymptotic geometry of nullhomotopies is rationally invariant, and section 6
uses this as well as the result of section 4 to prove the main theorem.
Finally, in section 7 we discuss lower bounds on the size of homotopies.

\subsection{Acknowledgments}
The authors would like to thank Dominic Dotterrer for many useful conversations
and suggestions over the course of the development of this project, and Mike
Freedman for insightful questions that we believe helped improve the exposition.
We are grateful to the anonymous referee for many corrections and suggestions for
clarifying the exposition.  The second author would like to thank Alex Nabutovsky
for pointing out certain prior work and related problems.

The first author was partially supported by an NSERC postdoctoral fellowship.

%{\it [Other helpful contributions?  Other sources of support? --GC]}

%%%%%%%%%%%%%%%%%%%%%%%%%%%%%%%%%%%%%%%%%%%%%%%%%%%%%%%%%%%%%%%%%%%%%%%%%%%%%%%
\section{Preliminaries}
%%%%%%%%%%%%%%%%%%%%%%%%%%%%%%%%%%%%%%%%%%%%%%%%%%%%%%%%%%%%%%%%%%%%%%%%%%%%%%%

In this section, we summarize the geometric machinery developed in \cite{cob}
as well as introducing some of our own.

\subsection{Simplicial approximation and mosaic maps}
One result we will make heavy use of is a quantitative simplicial approximation
theorem, allowing us to approximate any map between simplicial complexes by a
simplicial one with a similar Lipschitz constant.  First, we need to define the
appropriate kind of subdivision.
\begin{defn}
  Define a \emph{simplicial subdivision scheme} to be a family, for every pair of
  natural numbers $n$ and $L$, of metric simplicial complexes $\Delta^n(L)$
  isometric to the standard $\Delta^n$ with length 1 edges, such that
  $\Delta^n(L)$ restricts to $\Delta^{n-1}(L)$ on all faces.  A subdivision scheme
  is \emph{regular} if for each $n$ there is a constant $A_n$ such that
  $\Delta^n(L)$ has at most $A_n$ isometry classes of simplices and a constant
  $r_n$ such that all 1-simplices of $\Delta^n(L)$ have length in
  $[r_n^{-1}L^{-1},r_nL^{-1}]$.

  Given a regular subdivision scheme, we can define the
  \emph{$L$-regular subdivision} of any metric simplicial complex, where each
  simplex is replaced by an appropriately scaled copy of $\Delta^n(L)$.
\end{defn}
Note that $L$ times barycentric subdivision is \emph{not} regular.  Two known
examples of regular subdivision schemes are the edgewise subdivision described
in \cite{EdGr} and the cubical subdivision described in \cite{FWPNAS}.
\begin{prop}[Quantitative simplicial approximation theorem]
  \label{prop:quantitative_simplicial_approximation}
  For finite simplicial complexes $X$ and $Y$ with piecewise linear metrics,
  there are constants $C$ and $C^\prime$ such that any $L$-Lipschitz map
  $f:X \to Y$ has a $CL$-Lipschitz simplicial approximation via a homotopy of
  thickness $CL+C^\prime$ and width $C^\prime$.
\end{prop}
As in \cite{cob}, we will use simplicial approximation mainly as a way of
ensuring that our maps have a uniformly finite number of possible restrictions
to simplices.  The property that we really care about, then, is the following:
\begin{defn}
  Let $\mathcal{F}_k$ be a finite set of maps $\Delta^k \to Y$, for some space
  $Y$.  If $X$ is a simplicial complex, a map $f:X \to Y$ is
  \emph{$\mathcal{F}$-mosaic} if all of its restrictions to $k$-simplices are
  in $\mathcal{F}_k$.  More generally, we can take $X$ to be any polyhedral
  complex with a finite collection of cell shapes (e.g.\ a cubical complex, or a
  product of simplicial complexes) and $\mathcal{F}_k$ to be a set of maps from
  each of the various shapes.

  We refer to a collection of maps as \emph{uniformly mosaic} if they are all
  $\mathcal{F}$-mosaic with respect to a fixed unspecified $\mathcal{F}$.
\end{defn}
The $\mathcal{F}_k$ in the definition naturally form a semi-simplicial set
$\mathcal{F}$ via restriction maps.  Thus we can think of an
$\mathcal{F}$-mosaic map $f$ equivalently as one that factors through
$$X \xrightarrow{g} \mathcal{F} \xrightarrow{h_{\mathcal{F}}} Y,$$
where $h_{\mathcal{F}}$ is fixed and $g$ is simplicial, or more generally takes
cells isomorphically to cells.  In particular, the property of a collection of
maps being uniformly mosaic is preserved by postcomposition with any map, for
example one collapsing certain simplices.

%% It's important to note that we can always add cells to $\mathcal{F}$ to make it
%% homotopy equivalent to $Y$.  Thus we will assume that it is in fact homotopy
%% equivalent.

\subsection{Isoperimetry for cochains}
In rational homotopy theory, algebraic nullhomotopies are constructed by
antidifferentiating certain exact differential forms.  To imitate this
construction geometrically, we need to be able to antidifferentiate simplicial
cochains in a quantitative way.  This is given to us by the following lemma,
proven in \cite{cob}.  Here, the $\ell^\infty$ norm of a cochain is simply the
maximum of its values on simplices.
\begin{lem}[$\ell^\infty$ coisoperimetry] \label{lem0'}
  Let $X$ be a finite simplicial complex equipped with the standard metric, and
  let $X_L$ be the cubical or edgewise $L$-regular subdivision of $X$, and $k
  \geq 1$.  Then there is a constant $C_{\mathrm{IP}}=C_{\mathrm{IP}}(X,k)$ such that
  for any simplicial coboundary $w \in C^k(X_L;\mathbb{Z})$, there is an
  $a \in C^{k-1}(X_L;\mathbb{Z})$ with $\delta a=w$ such that
  $\lVert a \rVert_\infty \leq C_{\mathrm{IP}}L\lVert w \rVert_\infty$.
\end{lem}
The proof of this fact uses the following lemma which we will also need
independently.
\begin{lem} \label{Zapprox2}
  With the same assumptions, there is a constant $K(X,k)$ such that for any real
  simplicial cocycle $w \in C^{k-1}(X_L,\mathbb{R})$, there is an integral cocycle
  $\tilde w \in C^{k-1}(X_L;\mathbb{Z})$ with $\lVert w-\tilde w \rVert_\infty \leq
  K$.
\end{lem}

\subsection{Quantitative De Rham theory} \label{QDR}
In order to prove the main theorem, we need to discuss cup products on the
cochain level.  Since simplicial cup products do not have particularly nice
properties, it will be more convenient to use differential forms.  Therefore it
will be helpful to be able to associate to each simplicial cochain a
corresponding standard differential form.  We use the notation $\int\omega$ to
denote the simplicial cochain obtained by integrating a differential form; here
we construct a chain homotopy inverse to this operation.

To do this, we use Whitney's proof of the De Rham theorem, provided in
\cite[\S IV.27]{GIT}.  Whitney constructs an explicit isomorphism $D_\bullet$
from the simplicial cochain complex $C^\bullet(M;\mathbb{R})$ of a manifold $M$
to a subcomplex of $\Omega^\bullet(M)$.  The same construction produces smooth
forms on any simplicial complex as a stratified space.  For every $n$, let
$\{g^n_i:0 \leq i \leq n\}$ be a smooth partition of unity on the standard
simplex $\Delta^n=\left\{\vec x:\sum_{j=0}^n x_j=1\right\} \subset
\mathbb{R}^{n+1}$ with the following properties:
\begin{itemize}
\item $g^n_i \equiv 1$ near the $i$th vertex and 0 near the opposite face;
\item $\{g^n_i\}$ is invariant under the action of the symmetric group;
\item for every $j$, $g^n_i$ is independent of $x_j$ when $x_j<\epsi_n$, for fixed
  $\epsi_n>0$;
\item $g^n_i|\Delta^{n-1}=g^{n-1}_i$.
\end{itemize}
On any simplicial complex $X$, this defines a smooth partition of unity $\{g_v:v
\in X^{(0)}\}$.  For a given $\ell$-simplex $c=(v_0,\ldots,v_\ell)$, Whitney then
defines
$$D_\ell(\chi_c)=\ell!\sum_{i=0}^\ell (-1)^ig_{v_i}dg_{v_0} \wedge \cdots \wedge
\widehat{dg_{v_i}} \wedge \cdots \wedge dg_{v_\ell}$$
($\phi_\ell$ in his notation) and shows that this induces a map
$D_\bullet:C^\bullet(M;\mathbb{R}) \to \Omega^\bullet(M)$ which is an isomorphism of
cochain complexes onto its image.

In order to apply this isomorphism to our situation, we need to make a few more
remarks:
\begin{enumerate}
\item If $p:X \to Y$ is a simplicial map, then $p^*D_\bullet=D_\bullet p^*$.  This
  follows from the special case of an $(n+1)$-simplex collapsed onto an
  $n$-simplex, which is itself easy to see.
\item Given two simplicial complexes $X$ and $Y$, the map
  $$D_\bullet^{X \times Y} := \pi_1^*D_\bullet^X \wedge \pi_2^*D_\bullet^Y:
  C^\bullet(X;\mathbb{R}) \otimes C^\bullet(Y;\mathbb{R}) \to
  \Omega^\bullet(X \times Y)$$
  is likewise an isomorphism from the cellular cochains on the product cell
  structure on $X \times Y$ to its image.  In particular, we will use this in
  the setting $Y=[0,1]$, split into some number of 1-simplices.  We will say a
  form is \emph{desimplicial} if it is in the image of this map.
\item Given a Riemannian metric on each stratum of $X$ and a form $\omega$,
  define $\lVert\omega\rVert_\infty$ to be the maximum value of $\omega$ on a
  tuple of unit vectors.  Then there are constants $C_\ell$ such that if we put
  on $X \times Y$ the product metric of the standard metrics on simplices, then
  $$\lVert D_\ell(c) \rVert_\infty \leq C_\ell\lVert c \rVert_\infty.$$
\item Let $\omega \in \Omega^*(X \times [0,1])$ be a desimplicial form with
  $\omega|_{X \times \{1\}} \equiv 0$.  For a multivector $\xi \in T^n_{(x,t)}(X
  \times I)$, write $\xi_s$ for the corresponding multivector in $T^n_{(x,s)}(X
  \times I)$.  Then the form
  $$\alpha(\xi)=\int_t^1 \omega_i(\xi_s,ds)ds$$
  is also desimplicial.
\end{enumerate}

%%%%%%%%%%%%%%%%%%%%%%%%%%%%%%%%%%%%%%%%%%%%%%%%%%%%%%%%%%%%%%%%%%%%%%%%%%%%%%
\section{The case of maps $S^3 \to S^2$}
%%%%%%%%%%%%%%%%%%%%%%%%%%%%%%%%%%%%%%%%%%%%%%%%%%%%%%%%%%%%%%%%%%%%%%%%%%%%%%

In this section, as a warmup, we handle a concrete special case which touches
upon most of the problems which we will encounter in proving the more general
theorem.
\begin{thm}
  There is a constant $C$ such that any nullhomotopic $L$-Lipschitz map $f:S^3
  \to S^2$ has a nullhomotopy of width $C(L+1)^2$ and thickness $C(L+1)$.
\end{thm}
\begin{proof}
  We first give our spaces some extra structure.  We embed $S^2$ in
  $\mathbb{C}\mathbf{P}^2$, giving each the cell structure with one cell in each
  even dimension.  We give $S^3$ a simplicial structure which is an $L$-regular
  subdivision of some standard one, for example that of $\partial\Delta^4$.
  Finally, let $I=[0,1]$ be given the simplicial structure with
  $C_{\mathrm{IP}}(S^2,2)L^2$ edges of equal length.

  By postcomposing a simplicial approximation with a map contracting simplices,
  and at the cost of a multiplicative increase in $L$, we can assume that maps
  $f:S^3 \to S^2$ are cellular and uniformly mosaic, with restrictions to
  2-simplices having degree between $-1$ and $1$.  We now use a construction
  similar to that of \cite[Thm.~4.2]{cob} to construct a complex $\mathcal{G}$
  (independent of $f$) and a $\mathcal{G}$-mosaic nullhomotopy
  $$F:S^3 \times I \to \mathcal{G} \to \mathbb{C}\mathbf{P}^2.$$
  Since $f$ is cellular, we can define a cochain $w \in C^2(S^3;\pi_2(S^2))$ by
  $\langle w,c \rangle=[f|_c] \in \pi_2(S^2)$.  Since $f$ is nullhomotopic, this
  cochain is the coboundary of some $a \in C^1(S^3;\pi_2(S^2))$.  By Lemma
  \ref{lem0'}, since $\lVert w \rVert_\infty=1$, we can pick $a$ such that
  $\lVert a \rVert_\infty \leq C_{\mathrm{IP}}L$.

  Now let $\hat a \in C^1(S^3 \times I;\pi_2(S^2))$ be defined by
  $$\begin{aligned}
    \left\langle \hat a, v \times
    \left[\frac{i}{C_{\mathrm{IP}}L^2},\frac{i+1}{C_{\mathrm{IP}}L^2}\right]
    \right\rangle &=0 &
    \begin{array}{r}
      \text{for $0$-simplices $v$ of }S^3,\\
      0 \leq i \leq C_{\mathrm{IP}}L^2;
    \end{array} \\
    \left\langle \hat a, e \times \left\{\frac{i}{C_{\mathrm{IP}}L^2}
    \right\}\right\rangle &
    =\left\lfloor \left(1-\frac{i}{C_{\mathrm{IP}}L^2}\right)^2\langle a,e \rangle
    \right\rfloor &
    \begin{array}{r}
      \text{for $1$-simplices $e$ of }S^3,\\
      0 \leq i \leq C_{\mathrm{IP}}L^2.
    \end{array}
  \end{aligned}$$
  In other words, $\hat a$ is the ``rounded off'' version of the cochain $\bar a$
  whose value on $e \times \{t\}$ is $(1-t^2)\langle a,e\rangle$.  This ensures
  that the cochain $\delta\hat a$ has the following properties:
  \begin{enumerate}
  \item $\lVert \delta\hat a \rVert_\infty \leq 3$;
  \item $\delta\hat a|_{S^3 \times \{0\}}=w$;
  \item and $\delta\hat a|_{S^3 \times \{1\}}=0$.
  \end{enumerate}
  This allows us to build $F:S^3 \times I \to \mathbb{C}\mathbf{P}^2$ by skeleta
  as follows.  Send the $1$-skeleton to the basepoint; this gives us
  $\mathcal{G}_0$ and $\mathcal{G}_1$ with one map each.  Then send each 2-cell
  $c$ to $S^2 \subset \mathbb{C}\mathbf{P}^2$ via a fixed map of degree
  $\langle \delta\hat a, c \rangle$.  This gives us a $\mathcal{G}_2$ with one
  cell per degree between $-3$ and $3$ and shape of cell, and a partial map
  $F:(S^3 \times I)^{(2)}$ which can be extended to the 3-skeleton with no
  obstruction since $\delta\hat a$ evaluates to zero on cycles.  For each
  possible map on the boundary of a 3-cell, we fix a filling, giving
  $\mathcal{G}_3$ and an extension of $F$ to the 3-skeleton.  Since there is no
  obstruction to extending the map to the 4-skeleton, we again fix a filling for
  each possible map on the boundary of a 4-cell.  At each step, we also include
  the zero map and the restrictions of $f$ to simplices in the corresponding
  skeleton of $\mathcal{G}$, and ensure that the restriction to
  $S^3 \times \{0,1\}$ is correct.  This completes the construction of $F$.

  We will proceed by changing this nullhomotopy into one which maps to $S^2$.
  For this to work, we need to kill the Hopf obstruction; that is, to change its
  behavior on the 3-cells of $S^3 \times I$ so that the restriction to the the
  boundary of each 4-cell of $S^3 \times I$ has Hopf invariant 0.

  Let us translate this into the language of differential forms.  The cohomology
  ring $H^*(\mathbb{C}\mathbf{P}^2;\mathbb{Z})=\mathbb{Z}[x]/(x^3)$, where $x
  \in H^2(\mathbb{C}\mathbf{P}^2;\mathbb{Z})$.  Let $\xi$ be a differential form
  representing $x$, with the extra property that $f^*\xi$ is desimplicial; this
  is possible from the restrictions we put on $f$.  Then for a 4-cell $p$ of
  $S^3 \times I$, $\int_p F^*\xi^2$ is the degree of $F|_p$ over the 4-cell of
  $\mathbb{C}\mathbf{P}^2$, or equivalently the Hopf invariant of $F|_{\partial p}$
  (this restriction is a map to $S^2$ since $F$ is cellular.)  If $\alpha$ is any
  1-form with $d\alpha=F^*\xi$, then this Hopf invariant is given by Stokes'
  theorem by $\int_{\partial p} \alpha \wedge F^*\xi$.  Now suppose we have a
  cochain $b \in C^3(S^3 \times I;\mathbb{Z})$ such that
  $\langle b,\partial p\rangle$, or in other words $\delta b=\int F^*\xi^2$, but
  which (probably unlike $\int \alpha \wedge F^*\xi$) takes uniformly bounded,
  integer values on simplices.  This would allow us to construct the new
  nullhomotopy as follows.  Given two maps $u,v$ from a disk (of any dimension
  $m$) to some other space which coincide on the boundary, let $u * v$ denote the
  map on the $m$-sphere which restricts to $u$ and $v$ on the two hemispheres.
  Then:
  \begin{itemize}
  \item For each 3-cell $q \in S^3 \times I$, replace $F|_q$ with a map $G|_q$
    such that the map $F|_q*G|_q:S^3 \to S^2$ (which behaves like $F$ on the
    upper hemisphere and $G$ on the lower) has Hopf invariant
    $\langle b, q \rangle$.
  \item Extend $G$ to 4-cells; this can be done since the Hopf invariant on the
    boundary of each 4-cell is zero.
  \end{itemize}
  Finding a $b$ which satisfies these properties will be the goal of the rest of
  the proof.

  We note that the behaviors of $F^*\xi$ on $k$-cells are in one-to-one
  correspondence with the set $\mathcal{G}_k$.  For now, though, instead of
  $F^*\xi$ we will use the desimplicial form $\hat\omega:=D_2\int F^*\xi$.  This
  allows us to define a nice antidifferential.

  We write $\omega \in \Omega^2(S^3)$ to mean the restriction of $\hat\omega$ to
  $S^3 \times \{0\}$ (which is also $f^*\xi$.)  Further on, we will also define a
  ``smooth'' interpolation $\bar\omega$ between $\omega$ and $0$, as opposed to
  the ``bumpy'' interpolation $\hat\omega$.  Note also that $\omega$ is the
  ``differential form version'' of the cochain $w$.  We use a similar convention
  for other forms further on.

  For a vector $v \in T_{(x,t)}(S^3 \times I)$, write $v_s$ for its translate in
  $T_{(x,s)}(S^3 \times I)$.  Now, since $\hat\omega|_{S^3 \times \{1\}} \equiv 0$,
  and by the Poincar\'e lemma, the 1-form
  $$\hat\alpha(v)=\int_t^1 \hat\omega(ds,v_s)ds \in \Omega^1(S^3 \times I)$$
  satisfies $d\hat\alpha=\hat\omega$ and $\int\hat\alpha=\hat a$.  Moreover,
  since $\hat\omega$ is desimplicial, this also means that
  $\lVert\alpha\rVert_\infty \leq L\lVert\omega\rVert_\infty$ and, as discussed in
  \S\ref{QDR}, $\alpha$ is desimplicial.  Thus we know that
  $\lVert\hat\alpha\wedge\hat\omega\rVert_\infty \leq CL$, but we don't have a
  constant bound.  On the way to defining the desired cochain $b$, we will find a
  uniformly bounded form $\beta$ such that for any 4-cell $p$,
  $$\int_{\partial p} \beta=\int_{\partial p} \hat\alpha \wedge \hat\omega=
  \int_p \hat\omega^2.$$
  Even then, we will not be able to simply set $b=\int\beta$, both because
  $\int\beta$ may not be integral and because $\delta\int\beta=\int\hat\omega^2$,
  which is potentially a different cochain from $\int F^*\xi^2$.  Nevertheless,
  after $\beta$ is constructed, there is only a short way to go to building $b$.

  To construct $\beta$, we recall the algebraic nullhomotopy
  $\bar h:\mathcal{M}^*(S^2) \to A^*(S^3) \otimes \mathbb{Q}\langle t,dt \rangle$
  from the introduction, given by
  \begin{align*}
    x &\mapsto \omega \otimes (1-t)^2-\alpha \otimes 2(1-t)dt \\
    y &\mapsto \eta \otimes 4(1-t)^3dt,
  \end{align*}
  where we choose $\alpha$ and $\eta$ so that $\alpha=\hat\alpha|_{S^3 \times \{0\}}$
  and $d\eta=\alpha \wedge \omega$.  Note that our isoperimetric results mean
  that we can choose $\eta$ to have $\infty$-norm $O(L^2)$.  Moreover, since we
  have subdivided the interval into $O(L^2)$ pieces, $dt$ thought of as a 1-form
  on this subdivision has $\infty$-norm $O(L^{-2})$.  Now, $\hat\omega^2$ can be
  thought of as an approximation of
  $$\bar h(x^2)=d\bar h(y)=\alpha \wedge \omega \otimes 4(1-t)^3dt.$$
  Therefore we can use the bounded form $\bar h(y)=\eta \otimes 4(1-t)^3dt$ as a
  scaffolding to help us build a form with bounded $\infty$-norm whose derivative
  is $\hat\omega^2$.

  To this end, writing $\pi:S^3 \times [0,1] \to S^3$ for the projection onto
  the first factor, we let
  \begin{align*}
    \Delta\alpha &:= \hat\alpha-\bar\alpha,\text{ where }\bar\alpha=
    (1-t)^2\pi^*\alpha \\
    \Delta\omega &:= d\Delta\alpha=\hat\omega-\bar\omega,\text{ where }
    \bar\omega=(1-t)^2\pi^*\omega-2(1-t)dt \wedge \pi^*\alpha.
  \end{align*}
  In other words, $\Delta\alpha$ is the (bounded!) difference between
  $\hat\alpha$ and the form that $\hat\alpha$ would be if we hadn't had to take
  integer parts in the construction of its cochain counterpart $\hat a$.  So by
  construction, $\Delta\alpha$ and $\Delta\omega$ are both bounded.

  Now, by Stokes' theorem, for any $4$-cell $p$ of $S^3 \times I$, the Hopf
  invariant of $F$ on its boundary is given by
  \begin{align*}
    \int_p \hat\omega^2 &= \int_p \left[(\Delta\omega)^2
    +2\bar\omega \wedge \Delta\omega+\bar\omega^2\right] \\
    &= \int_p \left[(2\hat\omega-\Delta\omega) \wedge \Delta\omega
      -4(1-t)^3dt \wedge \pi^*(\alpha \wedge \omega)\right] \\
    &= \int_{\partial p} \left[(2\hat\omega-\Delta\omega) \wedge \Delta\alpha
      -4(1-t)^3dt \wedge \pi^*\eta\right].
  \end{align*}
  Here, the equality between the first and second lines holds because $\alpha
  \wedge \alpha$ and $\omega \wedge \omega$ are both zero.

  Call the integrand in the last line $\beta$.  We see that both terms of $\beta$
  are uniformly bounded and are zero when restricted to $S^3 \times \{0,1\}$.

  Now consider the uniformly bounded cellular cochain $\int\beta \in
  C^3(S^3 \times I;\mathbb{R})$.  We have $\delta\int\beta=\int\hat\omega^2$, but
  it may not be the case that $\int \hat\omega^2$ is the same cochain as
  $\int F^*\xi^2$, which is the degree of $F$ on 4-cells.  This can be resolved
  in the following manner.  Recall that $F$ factors through maps
  $$S^3 \times I \xrightarrow{G} \mathcal{G} \xrightarrow{H}
  \mathbb{C}\mathbf{P}^2,$$
  where $\mathcal{G}$ is a fixed finite polyhedral complex independent of $L$.
  Then $\left(D_2\int H^*\xi\right)^2$ and $H^*\xi^2$ are well-defined,
  cohomologous forms on $\mathcal{G}$ and thus there is a cellular cochain
  $\mathfrak{g} \in C^3(\mathcal{G};\mathbb{R})$, again independent of $L$, such
  that
  $$d\mathfrak{g}={\textstyle\int} \left(D_2{\textstyle\int} H^*\xi\right)^2-
  {\textstyle\int} H^*\xi^2.$$
  Then $b'=\int\beta-G^*\mathfrak{g}$ is a uniformly bounded cochain on
  $S^3 \times I$ with $\delta(\int\beta-G^*\mathfrak{g})=\int F^*\xi^2$.  This
  cochain is not integral, but it does restrict to zero on $S^3 \times \{0,1\}$.

  By Lemma \ref{Zapprox2}, we can find an integral cochain
  $b_0 \in C^2(S^3;\mathbb{Z})$ such that for every 2-simplex $q$ of $S^3$,
  $$|\langle b_0, q\rangle-\langle b',q \times [0,1]\rangle| \leq K(S^3,3).$$
  We then set $b$ by taking nearest integers to $b'$, similarly to how we
  constructed $\hat a$ from $\bar a$.  Specifically, we set
  $\langle b, q \times [(i-1)/C_{\mathrm{IP}}L^2,i/C_{\mathrm{IP}}L^2]\rangle$ so that
  $$\langle b-b',q \times [0,i/C_{\mathrm{IP}}L^2]\rangle \in [0,1),$$
  for $i \neq C_{\mathrm{IP}}L^2$ (these values are at most distance 1 from those of
  $b'$) and set the values on the last time increment so that
  $\langle b,q \times [0,1]\rangle=\langle b_0,q \rangle$ (and hence they are at
  most $K(S^3,3)+1$ away from those of $b'$).  This together with the requirement
  that $\delta b=\int F^*\xi^2$ fixes the values on the transverse 3-simplices of
  $S^3 \times [0,1]$; these values are at most distance 4 from those of $\beta$.
  Therefore we get
  $$\lVert b \rVert_\infty \leq \lVert b' \rVert_\infty+\max\{K(S^3,3)+1,4\}.$$
  This is a uniform bound and completes the proof.
\end{proof}

%%%%%%%%%%%%%%%%%%%%%%%%%%%%%%%%%%%%%%%%%%%%%%%%%%%%%%%%%%%%%%%%%%%%%%%%%%%%%%
\section{Lifting through $k$-invariants}
%%%%%%%%%%%%%%%%%%%%%%%%%%%%%%%%%%%%%%%%%%%%%%%%%%%%%%%%%%%%%%%%%%%%%%%%%%%%%%

We now extend the argument for $S^3 \to S^2$ to a setting which is still
geometrically constrained, but which contains a larger class of rational homotopy
types which, together with the rational invariance results in the next section,
can be assembled into the final result.
\begin{thm} \label{thm:lvl2}
  Let $X$ be a finite $N$-dimensional simplicial complex.  For $i=1,\ldots,r$,
  let $n_i \geq 2$ and let $B_i$ be a finite CW complex with an $(N+1)$-connected
  map $B_i \to K(\mathbb{Z},n_i)$ whose CW structure is that of a simplicial
  complex with the $(n_i-1)$-skeleton collapsed.  Define a CW-complex
  $B=\prod_{i=1}^r B_i$.  For some $2 \leq n \leq N$, let $Y$ be a finite
  subcomplex, whose inclusion map is $(N+1)$-connected, of the total space of a
  $K(\mathbb{Z},n)$-fibration over $B$, with projection map $p:Y \to B$.  Then
  there is a $C(X,Y)$ such that any nullhomotopic $L$-Lipschitz map $f:X \to Y$
  has a nullhomotopy of width $C(L+1)^2$ and thickness $C(L+1)$.
\end{thm}
Note that $K(\mathbb{Z},n)$-fibrations over $B$ are, up to equivalence, in
bijection with elements of $H^{n+1}(B)$ which represent the obstruction to
constructing a section, and that any such fibration can be made to have finite
skeleta, for example using Milnor's construction \cite{Milnor}.

The proof follows an outline similar to the special case in the previous
section.  The main differences are technicalities imposed by the need to lift
through a fibration where in the last section we retracted.
\begin{proof}
  Up to dimension $N$, $H^*(B;\mathbb{Q})$ is naturally isomorphic to a free
  graded commutative $\mathbb{Q}$-algebra generated by elements of degree $n_i$.
  Suppose first that the primary (and only) obstruction in $H^{n+1}(B;\mathbb{Q})$
  to trivializing $p$ has an indecomposable summand in this algebra.
  Equivalently, $\pi_n$ of the fiber goes to a finite quotient in $Y$, so up to
  rational homotopy type up to dimension $N$, $Y$ is still a product of
  Eilenberg--MacLane spaces.  This case follows directly from the main theorem of
  \cite{cob}, and linear nullhomotopies can be found; therefore, in the rest of
  this proof, we assume that this obstruction class is contained in the ideal
  generated by products in $H^*(B;\mathbb{Q})$.

  We start by showing that $f$ can be assumed to take on a certain structure, in
  particular being uniformly mosaic on a subdivision of $X$ at scale $L$.  We
  will implicitly work with such a subdivision; when we take the $L^\infty$ norm
  of forms on $X$, we will do so with respect to the metric in which the
  simplices of the subdivision are of unit size.

  Let $\pi_i$ be the projection $B \to B_i$.  We can simplicially approximate a
  map homotopic to $\pi_i \circ p \circ f$ on the distinguished simplicial model
  of $B_i$, then send it back to $B_i$ via the map collapsing the
  $(n_i-1)$-skeleton.  This gives us a short homotopy between $\pi_i \circ p
  \circ f$ and a cellular map on an $O(L)$-regular subdivision of $X$ which is
  $\mathcal{F}^i$-mosaic for some $\mathcal{F}^i$ depending only on $B_i$ and
  the homotopy equivalence.  This gives us a short homotopy $E_t$ from
  $p \circ f$ to a $\mathcal{F}$-mosaic map, where $\mathcal{F}_k=
  \prod_{i=1}^r \mathcal{F}^i_k$ and the boundary maps are also products.

  Finally, we would like to lift $E_t$ to a short homotopy of $f$.  By homotopy
  lifting, this can be done, but we would like it done quantitatively in order
  to produce a short homotopy $\tilde E_t$ from $f$ to an
  $\tilde{\mathcal{F}}$-mosaic map for some $\tilde{\mathcal{F}}$.  Therefore we
  do this by skeleta.  For $k<n$, we can choose a unique lift for every
  $k$-simplex of $\mathcal{F}$.  Now let $c$ be an $n$-simplex of $X$.  We would
  like to show that we can lift $E_t|_c$ so that $\tilde E_1|_c$ is one of a
  uniformly finite number of maps.

  Let $\tilde c=c \times \{0\} \cup \partial c \times [0,1]$.  Since $E$ is
  uniformly Lipschitz with respect to the standard metric on the subdivision, we
  can simplicially approximate $E|_{\tilde c}$ at a uniform scale.  In particular,
  if $u:\tilde c \times [0,1]$ is the linear homotopy to the simplicial
  approximation, the map $u|_{\tilde c \times \{1\} \cup \partial \tilde c \times [0,1]}$
  takes on a uniformly finite number of values which we include in
  $\tilde{\mathcal{F}}_n$.  We can take this to be the map $\tilde E_1|_c$.

  Finally, for higher skeleta all lifts are again homotopic, and so when $k>n$
  we can take a unique lift for every restriction of $\tilde E_t$ to the
  boundary of a $k$-cell.  The set of such lifts will be called
  $\tilde{\mathcal{F}}_k$.

  Now, at the cost of a linear penalty on $L$, we can assume that $f$ is
  $\tilde{\mathcal{F}}$-mosaic, and therefore each $f_i$ is
  $\mathcal{F}^i$-mosaic.  For each $f_i$, a construction similar to that of the
  homotopy $F$ in the previous section builds a nullhomotopy
  $F_i:X \times I \to B_i$ to the following specifications.
  \begin{itemize}
  \item $F_i$ is $\mathcal{G}^i$-mosaic for some $\mathcal{G}^i$, again
    depending only on $B_i$, on a cell structure obtained by splitting the
    interval $I$ into $C_\ell^2L^2$ equal subintervals, where $C_\ell=
    \max_{n \leq N} C_{\mathrm{IP}}(X,n)$;
  \item The degree of $F_i$ on $n_i$-cells of this cell structure is as follows.
    Let $w_i \in C^{n_i}(X)$ be the cochain whose values are the degrees of $f_i$
    on simplices, and let $a_i \in C^{n_i-1}(X)$ be a cochain with $\delta a_i=w_i$
    and $\lVert a_i \rVert_\infty \leq C_{\mathrm{IP}}L\lVert w_i \rVert_\infty$.  Such
    an $a_i$ exists since $f_i$ is nullhomotopic.  Then the degree of $F_i$ on a
    cell $c$ is given by $\langle\delta\hat a_i,c\rangle$, where
    $\hat a_i \in C^{n_i-1}(X \times I)$ is defined by
    $$\langle \hat a_i, c \times \{t\} \rangle=
    \lfloor (1-t)^{n_i} \langle \hat a_i,c \rangle\rfloor$$
    on cells of that form and is zero on cells which extend in the time
    direction.  Since the derivative of $(1-t)^{n_i}$ and the values of $w_i$ are
    uniformly bounded, so are the values of $\delta \hat a_i$.
  \end{itemize}
  Then $F=(F_1,\ldots,F_r)$ is a nullhomotopy $F$ of $f$ in $B$ which is
  $\mathcal{G}$-mosaic, where once again $\mathcal{G}_k=
  \prod_{i=1}^r \mathcal{G}^i_k$.  Our plan is to find a nullhomotopy in $Y$ which
  projects onto $F$, again modeled on an algebraic nullhomotopy
  $$\bar h:\mathcal{M}^*(Y) \to \Omega^*(X) \otimes
  \mathbb{Q}\langle t,dt \rangle.$$
  The minimal model of $B$ is a free algebra $\mathcal{M}(B)$ with trivial
  differential on the $n_i$-dimensional generators $x_i$ corresponding to the
  fundamental class of each $B_i$.  The projection $p:Y \to B$ corresponds to an
  extension
  $$p^*:\mathcal{M}^*(B) \to \mathcal{M}^*(Y)=\mathcal{M}(B) \otimes
  \mathbb{Q}\langle y \rangle,$$
  where $dy=P(x_1,\ldots,x_r)$ is the aforementioned cohomological obstruction in
  $\mathbb{Q}\langle x_1,\ldots,x_r \rangle$ to finding a section of $Y$; this is
  a polynomial all of whose terms have total degree at least 2.

  For each $i$, let $\xi_i \in \Omega^{n_i}(B_i)$ be a form representing the
  fundamental class of $B_i$.  We then write
  $$\omega_i=D_{n_i}{\textstyle \int f^*p^*p_i^*\xi_i}.$$
  Note that since $f_i$ is a composition of a simplicial map and a collapse,
  $\omega_i$ is the pullback of a desimplicial form $\xi_i'$ representing the
  fundamental class in $H^{n_i}(K(\mathbb{Z},n_i))$.  Thus we can find a form
  $\nu=f^*\zeta$ where
  $$d\zeta=P(p_1^*\xi_1',\ldots,p_r^*\xi_r').$$
  (Further on, we will write this as $P(\overrightarrow{p_-^*\xi'})$.)  Since
  $\zeta$ doesn't depend on $f$ or $X$, $\lVert\nu\rVert_\infty$ is uniformly
  bounded, as are the $\omega_i$.  Similar to the previous section, we have that
  for an $(n+1)$-simplex $p$ of $X$,
  $$\int_{\partial p} \nu=\int_p P(\vec\omega) \neq
  \int_p P\bigl(\overrightarrow{f^*p^*p_{-}^*\xi}\bigr),$$
  but they differ by a small coboundary and we will later need to take this into
  account.

  We therefore get a homomorphism $\bar f:\mathcal{M}^*(Y) \to \Omega^*(X)$
  defined by $x_i \mapsto \omega_i$ and $y \mapsto \nu$.  Since $f$ is
  nullhomotopic, we can build the algebraic nullhomotopy $\bar h$ of $\bar f$ as
  follows.  For any DGA $A$ define an operator
  $\int_0^1:A \otimes \langle t,dt \rangle \to A$ by
  $${\textstyle\int_0^1} a \otimes t^i=0,
  {\textstyle\int_0^1} a \otimes t^idt=(-1)^{\deg a}\frac{a}{i+1}.$$
  Then send
  \begin{align*}
    x_i &\mapsto \omega_i \otimes (1-t)^{n_i}+(-1)^{n_i+1}\alpha_i \otimes
    n_i(1-t)^{n_i-1}dt \\
    y &\mapsto \nu \otimes (1-t)^{n+1}+\eta \otimes (n+1)(1-t)^ndt,
  \end{align*}
  where $\alpha$ and $\eta$ are chosen so that $d\alpha_i=\omega_i$ and
  $d\eta={\textstyle\int_0^1} \bar h(P(\vec x))+(-1)^{n+1}\nu$.

  Note that the terms of $\int_0^1 \bar h(P(\vec x))$ are each a product
  of some $\omega_i$'s together with one $\alpha_i$.  Since $\alpha_i$ may be
  chosen so that $\lVert\alpha_i\rVert_\infty=O(L)$, this means that $\eta$ may be
  chosen so that $\lVert\eta\rVert_\infty=O(L^2)$.

  On the other hand, define a form $\hat\omega_i=D_{n_i}\int F_i^*p_i^*\xi_i \in
  \Omega^{n_i}(X \times [0,1])$.  This gives us a homomorphism
  $\bar F:\mathcal{M}(B) \to \Omega^*(X \times [0,1])$.

  Write $\pi:X \times [0,1] \to X$ for the projection onto the first factor.
  For a multivector $\xi \in T^n_{(x,t)}(X \times I)$, and writing $\xi_s$ for its
  parallel translate in $T^n_{(x,s)}(X \times I)$, let
  $$\hat\alpha_i(\xi)=\int_t^1 \omega_i(ds,\xi_s)ds.$$
  Then $d\hat\alpha_i=\hat\omega_i$.  Now defining forms $\Delta\alpha_i$,
  $\Delta\omega_i$, $\bar\alpha_i$ and $\bar\omega_i$ by
  \begin{align*}
    \Delta\alpha_i &:= \hat\alpha_i-\bar\alpha_i
    := \hat\alpha_i-(1-t)^{n_i}\pi^*\alpha_i \\
    \Delta\omega_i &:= d\Delta\alpha_i = \hat\omega_i-\bar\omega_i \\
    &:= \hat\omega_i-(1-t)^{n_i}\pi^*\omega_i
    -n_i(1-t)^{n_i-1}dt \wedge \pi^*\alpha_i,
  \end{align*}
  we get $\lVert\Delta\omega_i\rVert_\infty \leq C(N,B)$ and
  $\lVert\Delta\alpha_i\rVert_\infty \leq C(N,B)$.

  Now, by Stokes' theorem, for every $(n+1)$-cell $c$ of $X \times [0,1]$,
  \begin{align*}
    \int_c \bar F(P(\vec x)) &= \int_c P\left(\overrightarrow{\Delta\omega}
                               -\overrightarrow{\bar\omega}\right) \\
    &= \int_c \left[\sum_i \Delta\omega_i
      \poly\left(\overrightarrow{\Delta\omega},\overrightarrow{\bar\omega}\right)
      -P\bigl(\vec{\bar\omega}\bigr)\right] \\
    &= \int_{\partial c} \left[\sum_i \Delta\alpha_i
      \poly\left(\overrightarrow{\Delta\omega},\overrightarrow{\bar\omega}\right)
      -(1-t)^{n+1}\pi^*\nu-(n+1)(1-t)^n\pi^*\eta \wedge dt\right].
  \end{align*}
  Call the integrand in the previous line $\beta$.  Then since
  $\lVert dt \rVert_\infty=1/C_\ell^2L^2$ and the polynomials we have elided can
  be chosen so as to depend only on $P$, $\beta$ satisfies
  $\lVert \beta \rVert_\infty \leq C(N,Y)$ and $d\beta=P(\vec\omega)$.

  Now we are ready to construct a lift $\tilde F:X \times I \to Y$ of $F$.
  Since $F$ is $\mathcal{G}$-mosaic, we can view it as a composition
  $$X \times I \xrightarrow{G} \mathcal{G} \xrightarrow{H} B.$$
  Now let $\mathcal{G}^\prime$ be the smallest complex which surjects onto
  $\mathcal{G}$ and such that $\tilde{\mathcal{F}}$ in turn injects into it,
  with the composition induced by the projection $\tilde{\mathcal{F}} \to
  \mathcal{F}$.  In particular, $(\mathcal{G}^\prime)^{(n-1)}=\mathcal{G}^{(n-1)}$
  but in the $n$-skeleton, some cells have a number of duplicates compared to
  $\mathcal{G}$.  Then there are obvious maps
  $$X \times I \xrightarrow{G^\prime} \mathcal{G}^\prime \xrightarrow{H^\prime} B.$$
  We can build a partial lift of $H^\prime$ to $Y$ by lifting each map in
  $\mathcal{G}_k^\prime$, for each $k \leq n$, using the lift in
  $\tilde{\mathcal{F}}$ where it exists.  This then gives us a map
  $\tilde H:(\mathcal{G}^\prime)^{(n)} \to Y$ and an obstruction cocycle
  $\mathfrak{o} \in C^{n+1}(\mathcal{G}^\prime,\tilde{\mathcal{F}};\mathbb{Z})$ to
  extending it to $(\mathcal{G}^\prime)^{(n+1)}$ which is independent of $f$.

  Now, $\mathfrak{o}$ and $P(D_2\int (H^\prime)^*p^*p_i^*\xi_i)$ are both
  representatives of the obstruction class in
  $H^{n+1}(\mathcal{G}^\prime,\tilde{\mathcal{F}};\mathbb{R})$ to lifting $H$ to a
  map $\mathcal{G} \to Y$.  Therefore, there is a cellular cochain $\mathfrak{a}
  \in C^n(\mathcal{G}^\prime,\tilde{\mathcal{F}};\mathbb{R})$ such that
  $$\delta\mathfrak{a}={\textstyle \int} P(D_2{\textstyle \int} (H^\prime)^*p^*p_i^*
  \xi_i)-\mathfrak{o}.$$
  Since
  $$\overline{F^*}P(\vec x)=(G^\prime)^*P(D_2{\textstyle \int}
  (H^\prime)^*p^*p_i^*\xi_i),$$
  the uniformly bounded cochain $\int \beta-(G^\prime)^*\mathfrak{a} \in
  C^n(X \times I;\mathbb{R})$ satisfies
  $$\delta({\textstyle\int\beta}-(G^\prime)^*\mathfrak{a})=
  (G^\prime)^*\mathfrak{o}.$$
  This cochain is not integral, but we can use the method in the previous section
  to find a nearby integral cochain $\mathfrak{b}$ with the same coboundary, and
  such that it is still zero on $X \times \{0,1\}$.

  We will use $\mathfrak{b}$ to construct a lift of $F$ to $Y$ which is
  $\tilde{\mathcal{G}}$-mosaic for a $\tilde{\mathcal{G}}=
  \tilde{\mathcal{G}}(X,Y)$ which we first construct.  Let $\tilde{\mathcal{G}}$
  contain $\tilde{\mathcal{F}}$ and for $k \leq n-1$ let $\tilde{\mathcal{G}}_k$
  consist of the $\tilde H$-lifts of $\mathcal{G}^\prime_k$.  Next, for every value
  $\gamma$ that may be taken by $\mathfrak{b}$ and every element
  $\delta \in \mathcal{G}^\prime_n \setminus \tilde{\mathcal{F}}_n$ we add in a
  lift $\ell(\delta,\gamma)$ which differs from the one in $\mathcal{G}^\prime_n$
  by $\gamma$.  Finally, for any $k>n$, any cell of $\mathcal{G}^\prime_k$, and any
  lift of its boundary, we add a single extended lift to $\tilde{\mathcal{G}}_k$
  if one exists.

  Now we modify $\tilde H \circ G^\prime$ to define a map
  $\tilde F|_{(X \times [0,1])^{(n)}}$: for every $n$-cell $c$, we let the map on $c$
  be $\ell(\tilde H \circ G^\prime|_c,-\mathfrak{b}(c))$.  This kills the
  obstruction, allowing us to lift further to construct our
  $\tilde{\mathcal{G}}$-mosaic map $\tilde F:X \times [0,1] \to Y$.
\end{proof}

%%%%%%%%%%%%%%%%%%%%%%%%%%%%%%%%%%%%%%%%%%%%%%%%%%%%%%%%%%%%%%%%%%%%%%%%%%%%
\section{Rational invariance}
%%%%%%%%%%%%%%%%%%%%%%%%%%%%%%%%%%%%%%%%%%%%%%%%%%%%%%%%%%%%%%%%%%%%%%%%%%%%%

In this section, we show that the difficulty of nullhomotoping $L$-Lipschitz
maps $X \to Y$ depends on $Y$ only up to rational homotopy type.  The proof of
this can be separated into a topological result and a metric result.

We start with the metric result, which is again proven in \cite{cob}.  It shows
that if a map $X \to Z$ is homotopically trivial relative to a subspace
$Y \subset Z$ whose relative homotopy groups are finite, then one can find such
a homotopic trivialization which is geometrically bounded.
\begin{lem} \label{lem:Qlift}
  Let $Y \subset Z$ be a pair of finite simplicial complexes such that
  $\pi_k(Z,Y)$ is finite for $k \leq n+1$.  Then there is a constant $C(n,Y,Z)$
  with the following property.  Let $(X,A)$ be a pair of (not necessarily finite)
  $n$-dimensional simplicial complexes and $f:(X,A) \to (Z,Y)$ a simplicial map
  which is homotopic rel $A$ to a map $g:X \to Y$.  Then there is a short
  homotopy rel $A$ of $f$ to a map $g^\prime$ which lands in $Y$ and is homotopic
  as a map into $Y$ to $g$.  By ``short'', we mean that it is $C$-Lipschitz under
  the standard metric on the product cell structure on $X \times [0,1]$.
\end{lem}
Note that the constant $C$ does not depend on $X$ and in particular on the
choice of a subdivision of $X$.  Thus if we consider Lipschitz and not just
simplicial maps from $X$ to $Y$, the width of the homotopy remains constant,
rather than linear in the Lipschitz constant.

We now move on to the topological portion of the discussion, in which we prove
Theorem \ref{intro:Qinv}.  First we state this result more precisely.
\begin{figure} 
	\centering
    \begin{tikzpicture}
      \draw[very thick] (-0.5,0) ellipse (0.5 and 1);
      \draw[very thick] (0,-1) arc (-90:90:0.5 and 1);
      \draw[very thick] (-0.5,-1) -- (0,-1) (-0.5,1) -- (0,1);
      \draw[very thick] (6,-1) arc (-90:90:0.5 and 1);
      \draw[very thick] (6.5,-1) arc (-90:90:0.5 and 1);
      \draw[very thick] (6,-1) -- (6.5,-1) (6.5,1) -- (6,1);
      \draw (0,-1) -- (6,-1);
      \draw (0,1) .. controls +(right:0.5) and
      +(left:0.5) .. (1.5,3) .. controls +(right:0.5) and
      +(left:0.5) .. (3,1) .. controls +(right:0.5) and
      +(left:0.5) .. (4.5,3) .. controls +(right:0.5) and
      +(left:0.5) .. (6,1);
      \node[below] at (-0.25,-1) {$X \times [0,1]$};
      \node[below] at (6.25,-1) {$X \times [0,1]$};
      \node at (3,0) {$K$};
    \end{tikzpicture}
	
    \caption{A camel.  Note the two collars which are isometric to some fixed
      simplicial structure on $X \times [0,1]$.}
  \end{figure}
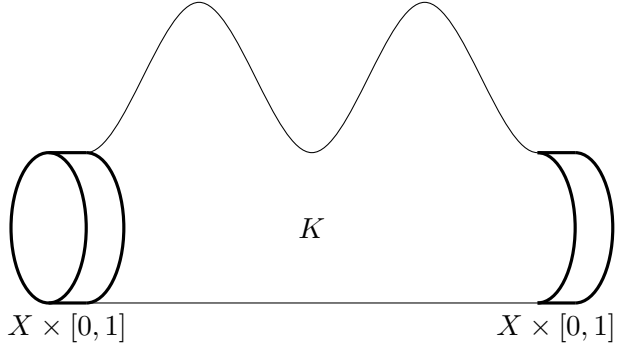
\begin{thm} \label{thm:Qinv}
  Rationally equivalent simply connected finite simplicial complexes admit
  \emph{nullhomotopies of the same shapes}.  That is, suppose we are given the
  following data:
  \begin{enumerate}
  \item Rationally homotopy equivalent simply connected finite metric simplicial
    complexes $Y$ and $Z$;
  \item	A finite $n$-dimensional simplicial complex $X$;
  \item	A simplicial pair $(K, X \times ([0,1] \cup [2,3]))$ which is
    homeomorphic to
    $$(X \times [0,3], X \times ([0,1] \cup [2,3]))$$
    and given the standard metric on simplices.  Here the product of $X$ with
    each unit interval is given an arbitrary fixed simplicial structure which
    restricts at $t=0$ and $t=1$ to the simplicial structure on $X$.
  \end{enumerate}
  Then there is a constant $C=C(X,Y,Z)>0$ such that if for every nullhomotopic
  $L$-Lipschitz map $f:X \rightarrow Y$ there is an $M$-Lipschitz nullhomotopy
  $F:K \rightarrow Y$, then for every $L/C$-Lipschitz map $g:X \rightarrow Z$
  there is a $(CM+C)$-Lipschitz nullhomotopy $G:K \rightarrow Z$.
\end{thm}
The point of introducing the complex $K$ is to prescribe a metric on the cylinder
$X \times [0,3]$.  The theorem then says that, under any such metric, sizes of
homotopies do not depend very much on torsion in the target space.  For example,
this controls the sizes of nullhomotopies through a camel with two humps, as in
the figure.  In the main application of this theorem to the proof of Theorem
\ref{intro:main}, $K$ will be a straight, but elongated cylinder whose length
depends on the the Lipschitz constant.
\begin{proof}
  Since $Y$ and $Z$ are rationally homotopy equivalent, there is a finite
  complex $W$ and a pair of maps $Y \to W \leftarrow Z$ which induce
  equivalences on the level of rational homotopy.   A proof for this is given,
  for example, in \cite[Lemma 1.3 and Cor.~1.9]{m}.  Thus we can assume that $Y$
  is a subcomplex of $Z$ or vice versa.

  We first do the case when $Y \subset Z$.  Let $C(n,Y,Z)$ be the constant given
  in Lemma \ref{lem:Qlift}.  Suppose $g:X \to Z$ is a nullhomotopic
  $L/C$-Lipschitz map, which we can assume to be simplicial on a subdivision
  $X_L$ at scale $\sim C/L$.  In particular, $g$ can be homotoped into $Y$, and
  so by Lemma \ref{lem:Qlift} this can be done via a short homotopy $H:X_L
  \times [0,1] \to Z$.  Now, $f(x):=H(x,1)$ is an $L$-Lipschitz nullhomotopic
  map $X \to Y$, and so there is an $M$-Lipschitz nullhomotopy $F:K \to Y$ of
  $f$.  Concatenating $H$ and $F$ gives an $(M+C)$-Lipschitz nullhomotopy
  $G:K \to Z$ of $g$.  This completes the first case.

  Now suppose $Z \subset Y$, and suppose $g:X \to Z$ is a nullhomotopic
  $L$-Lipschitz map.  By assumption, there is an $M$-Lipschitz nullhomotopy
  $F:K \to Y$ of $g$ (as a map to $Y$) and an uncontrolled nullhomotopy
  $\tilde G:CX \to Z$ of $g$.  There is no guarantee, however, that $F$ can be
  homotoped into $Z$ rel $X$, even in an uncontrolled way.

  On the other hand, concatenating $F$ and $\tilde G$ along $g$ gives us a map
  $H:SX \to Y$.  Homotopy classes of such maps form a group, and the induced
  mapping $[SX,Z] \to [SX,Y]$ is a homomorphism.  We would like to analyze the
  cokernel of this homomorphism; to do this, we use obstruction theory on the
  relative Postnikov tower
  \begin{center}
    \begin{tikzpicture}
      \node (x) at (0,0) {$Y$};
      \node (x0) at (3,0) {$P_1$};
      \node (x0prime) at (4.1,0) {$=P_0=Z$};
      \node (x1) at (3,1) {$P_2$};
      \node (vd) at (3,2) {$\vdots$};
      \node (xn) at (3,3) {$P_n$};
      \draw[->] (x) -- (x0) node[near end,anchor=south,inner sep=1pt]
           {$\ph_0=\ph$};
      \draw[->] (x) -- (x1) node[midway,anchor=south,inner sep=2pt] {$\ph_n$};
      \draw[->] (x) -- (xn) node[midway,anchor=south east,inner sep=1pt]
           {$\ph_n$};
      \draw[->] (xn) -- (vd) node[midway,anchor=west] {$p_n$};
      \draw[->] (vd) -- (x1) node[midway,anchor=west] {$p_3$};
      \draw[->] (x1) -- (x0) node[midway,anchor=west] {$p_2$};
    \end{tikzpicture}
  \end{center}
  of the inclusion $\ph:Y \hookrightarrow Z$.  Here, $P_k$ is a space such that
  $\pi_i(P_k,Y)=0$ for $i \leq k$ and $\pi_i(Z,P_k)=0$ for $i>k$.  The map $p_k$
  therefore only has one nonzero relative homotopy group, $\pi_k(Z,Y)$.  In this
  setting there is an obstruction theory long exact sequence of groups
  $$\cdots \to H^{k-1}(X;\pi_k(Z,Y)) \to [SX,P_k] \to [SX,P_{k-1}] \to
  H^k(X;\pi_k(Z,Y)) \to \cdots.$$
  Thus the cokernel we are interested in has cardinality at most
  $\prod_{i=1}^n \lvert H^k(X;\pi_k(Z,Y)) \rvert$.  For each element $\gamma$ of
  this cokernel, choose a map $F_\gamma:SX \to Y$ representing it.

  Now let $R \subset K=X \times (5/2,3]$.  Then there is an obvious $2$-Lipschitz
  homeomorphism $\psi_1:K \setminus R \to K$ which is the identity outside
  $X \times [2,3]$.  Also, let $\psi_2:\bar R \to SX$ be the surjection which
  contracts $X \times \{5/2\}$ and $X \times \{3\}$.  Then the map
  $$\tilde F(x)=\left\{\begin{array}{l l}
  F \circ \psi_1(x) & x \in K \setminus R \\
  F_{-[H]} \circ \psi_2(x) & x \in R
  \end{array}\right.$$
  gives us a nullhomotopy of $g$ which can be homotoped into $Z$ and which is
  $(C_0M+C_0)$-Lipschitz, where $C_0$ depends only on the geometry of the various
  $F_\gamma$.

  Finally, we can use Lemma \ref{lem:Qlift} to ensure that we get a
  $(CM+C)$-Lipschitz nullhomotopy $G:K \to Z$, where $C$ is the product of $C_0$
  and the constant from the lemma.
\end{proof}

%%%%%%%%%%%%%%%%%%%%%%%%%%%%%%%%%%%%%%%%%%%%%%%%%%%%%%%%%%%%%%%%%%%%%%%%%%%%%%
\section{Proof of the main theorem}
%%%%%%%%%%%%%%%%%%%%%%%%%%%%%%%%%%%%%%%%%%%%%%%%%%%%%%%%%%%%%%%%%%%%%%%%%%%%%%

Putting together Theorems \ref{thm:Qinv} and \ref{thm:lvl2}, we can now prove
Theorem \ref*{intro:main}.  We restate this theorem equivalently below:
\begin{thm*}
  Let $X$ be an $n$-dimensional finite complex, and let $Y$ be a finite complex
  which is rationally equivalent up to dimension $n$ to the total space of an
  induced fibration whose fiber and base are both products of simply connected
  Eilenberg--MacLane spaces.   Then there is a constant $C(X,Y)$ such that
  nullhomotopic $L$-Lipschitz maps from $X$ to $Y$ admit nullhomotopies of
  thickness $C(L+1)$ and width $C(L+1)^2$.
\end{thm*}
Let's unwrap a bit the rational homotopy theory of the spaces that we are
considering, particularly the word \emph{induced}. 
	
The basic fact that underlies everything is that $H^*(K(\mathbb{Q},k);\mathbb{Q})
=\mathbb{Q}[x]$ if $k$ is even and is $\mathbb{Q}[x]/(x^2=0)$ if $k$ is odd.
Note that both cases can be described as saying that the cohomology is the free
graded-commutative differential algebra on a $k$-dimensional class.  In light of
Kunneth’s theorem, we can now say that if $V$ is a graded vector space, and
$K(V)$ is a product of Eilenberg-MacLane spaces $K(V_k,k)$, then the rational
cohomology of $K(V)$ is the free graded algebra $\mathbb{Q}[V]$.
	
Notice of course, that if $X$ is a space whose rational cohomology is a free DGA,
then by considering the generating cohomology classes as maps into
Eilenberg--MacLane spaces, we get a map into a product of such spaces, i.e. a map
$X \to K(V)$ which is a rational isomorphism.
	
A special case is the even dimensional sphere.  $S^{2k} \to K(\mathbb{Z}, 2k)$ is
a tautological map.  However, the cup square vanishes for the sphere for
dimensional reasons, so this tautological map lifts naturally to the homotopy
fiber of the map $K(\mathbb{Z}, 2k) \xrightarrow{\cup 2} K(\mathbb{Z}, 4k)$.  The
map to this fiber is a rational equivalence, as seen using the fact that the
Euler class of the rational fibration sequence
$$S^{4k-1}_{(0)}=K(\mathbb{Q},4k-1) \to S^{2k}_{(0)} \to K(\mathbb{Q}, 2k)$$
is cup square together with the Gysin sequence.

More general homogeneous spaces $G/H$ have similar structure (after some work!)
The inclusion $H \to G$ is a homomorphism, and therefore induces a map 
$BH \to BG$ whose fiber is easily seen to be $G/H$.  For any connected Lie group
$K$, the cohomology is a free algebra, i.e., since $K$ is finite-dimensional, it is the cohomology of a product of odd spheres; by a theorem of Hopf, this product
is in fact homotopy equivalent to $K$ (see e.g.~Example 3 of
\cite[\S12(a)]{FHT}).  That the cohomology of $BK$ is also free is less obvious,
but is also classical; this cohomology can be described using the Lie algebra of
$K$.  This gives rise to a description of the map $BH \to BG$, which also shows
that the map $G/H \to BH$ is (up to homotopy) an induced fibration, a notion we
now explain in our setting.  A proof of this can be found in \cite[\S15(f)]{FHT}.

Suppose now that we have two graded $\mathbb{Q}$--vector spaces $V$ and $W$.  To
describe a map $f:K(V) \to K(W)$ is the same thing as describing a graded
homomorphism $W \to \mathbb{Q}[V]$.  The fiber $F$ of this map has a description
via a fibration
$$K(W, [-1]) \to F \to K(V).$$
(where the $[-1]$ indicates a shift in grading by $-1$), but it is not the most
general such fibration; we say, following \cite{Ganea}, that it is
\emph{induced} (by the map $K(V) \to K(W)$.  In the general case, the structure
group would be a space of self-homotopy equivalences of $K(W,[-1])$, but here we
are only allowing $K(W,[-1])$ itself, acting on itself as a topological group.
The classifying space of $K(W, [-1])$ is, of course $K(W)$.

In this case, the free algebra generated by $W$ with the shifted grading together
with $V$, equipped with the differential given by $dw=f^*w$, is a DGA model for
the space $F$.  A minimal model for this DGA is obtained by deleting pairs of
indecomposables (that is, elements of $W$ and $V$) $(g,h)$ with $dg=h$.
Conversely, given such a minimal model, we can construct an induced fibration
using the recipe above.  This shows that this condition is equivalent to that in
the introduction.

Moreover, by choosing a lattice $V_{\mathbb{Z}} \subset V$ and a lattice in $W$
whose differential lands in $\mathbb{Q}[V_{\mathbb{Z}}]$, one constructs an induced
fibration of this form whose homotopy groups are all free abelian and whose total
space has finite skeleta.  We use this construction in the proof below.

Thus the space $Y$ in Theorem \ref{intro:main} can be any simply-connected
homogeneous target space, including spheres, complex projective spaces, and
Grassmannians.  Another corollary concerns maps to spaces which are highly
connected.  The first part is a result from \cite{cob}.
\begin{cor} \label{cor:highly_connected}
  Let $Y$ be a rationally $(k-1)$-connected finite complex and $X$ an
  $n$-dimensional finite complex.
  \begin{enumerate}[label=(\alph*),leftmargin=2em]
  \item If $n \leq 2k-2$, then there is a constant $C(X,Y)$ such that homotopic
    $L$-Lipschitz maps from $X$ to $Y$ admit $C(L+1)$-Lipschitz homotopies.
  \item If $2k-1 \leq n \leq 3k-3$, then there is a constant $C(X,Y)$ such that
    nullhomotopic $L$-Lipschitz maps admit nullhomotopies of thickness
    $C(L+1)$ and width $C(L+1)^2$.
  \end{enumerate}
\end{cor}
\begin{proof}[Proof of Theorem \ref{intro:main}.]
  As discussed above, $Y$ is rationally equivalent up to dimension $n$ to the
  total space $Z$ of an induced fibration
  $$\prod_{j=1}^s K(\mathbb{Z},n_j) \to Z \to \prod_{i=1}^r K(\mathbb{Z},n_i).$$
  As noted before, we may assume that $Z$ is the fiber product of fibrations
  $K(\mathbb{Z},n_j) \to Z_j \to B$, where $B$ is as in Theorem \ref{thm:lvl2}
  and the $Z_j$ have finite skeleta.  Concretely, we can think of $Z$ as the
  pullback via the diagonal map $B \to B^s$ of the product fibration
  $$Z_1 \times \cdots \times Z_s \to B^s.$$
  Then if $f$ is a nullhomotopic map $X \to Z$, we can construct a nullhomotopy
  with the desired properties by finding a nullhomotopy $F$ in $B$, lifting it
  as in Theorem \ref{thm:lvl2} to $\tilde F_j:X \times I \to Z_j$ for each
  $1 \leq j \leq s$, and finally setting
  $$\tilde F(x,t)=(\tilde F_1(x,t), \ldots, \tilde F_s(x,t)) \in Z.$$
  Once we have shown the result for $Z$, it must hold for $Y$ by Theorem
  \ref{thm:Qinv}, as follows.  For a given $L$, we take $K_L=X \times [0,L]$.  We
  have shown that for any nullhomotopic $L$-Lipschitz $f:X \to Z$, there is a
  $CL$-Lipschitz nullhomotopy of $f$ in $K_L$.  Therefore, the same is true in
  $Y$, with a different constant.  Compressing $K_L$ back down to
  $X \times [0,1]$, we get back our separate estimates on thickness and width.
\end{proof}

%%%%%%%%%%%%%%%%%%%%%%%%%%%%%%%%%%%%%%%%%%%%%%%%%%%%%%%%%%%%%%%%%%%%%%%%%%%%%%
\section{Some lower bounds}
%%%%%%%%%%%%%%%%%%%%%%%%%%%%%%%%%%%%%%%%%%%%%%%%%%%%%%%%%%%%%%%%%%%%%%%%%%%%%%

It is first worth noting that nullhomotopies of maps, for example, from $S^n \to
S^n$ and $S^3 \to \mathbb{C}\mathbf{P}^2$ cannot be done in constant time, as is
the case for targets with finite homotopy groups as in Theorem 1 of
\cite{FWPNAS}.  Thus the linear upper bound in Conjecture \ref{conjNull} when
$q=1$ is sharp.  All this is discussed in \cite{cob}.

One may ask then whether Theorem \ref{intro:main} similarly gives sharp bounds.
This boils down to two separate questions.  First, is the quadratic bound on the
width of the homotopy necessary, or could a linear bound suffice?  Secondly, can
the theorem be extended to homotopies rather than just nullhomotopies, as is the
case with the theorem in \cite{cob}?  It turns out that both of these features
are required.

\subsection{Maps that are hard to nullhomotope}
First, we give a series of examples (see also \cite{cob}) that show that for
every $q$, the upper bound in Conjecture \ref{conjNull} is the best one possible
in general.  In particular, we show that it gives a sharp estimate on the minimum
volume of a nullhomotopy in certain cases; this can potentially be apportioned to
the width and thickness in other ways.  Let the space $Y_q$ be given by
$S^2 \vee S^2$ together with $(q+3)$-cells whose attaching maps form a basis for
$\pi_{q+2}(S^2 \vee S^2) \otimes \mathbb{Q}$.  Note that by rational homotopy
theory, $\pi_{*+1}(S^2 \vee S^2) \otimes \mathbb{Q}$ is a free graded Lie algebra
on two generators of degree 1 whose Lie bracket is the Whitehead product.  In
particular, if $f$ and $g$ are the identity maps on the two copies of $S^2$, the
iterated Whitehead product
$$h_1=[f,[f,\ldots[f,g]\ldots]]:S^{q+2} \to S^2 \vee S^2,$$
with $f$ repeated $q$ times, represents a nonzero element of
$\pi_{q+2}(S^2 \vee S^2)$.  Moreover, the map
$$h_L=[L^2f,[L^2f,\ldots[L^2f,L^2g]\ldots]]:S^{q+2} \to S^2 \vee S^2$$
is an $O(L)$-Lipschitz representative of $L^{2q+2}[h_1]$.  Thus in $Y_q$, we can
define a nullhomotopy $H$ of $h_L$ by first homotoping it inside
$S^2 \vee S^2$ to $h_1 \circ \ph_{2q+2}$ for some map $\ph_{2q+2}:S^{q+2} \to S^{q+2}$
of degree $L^{2q+2}$, and then nullhomotoping each copy of $h_1$ via a standard
nullhomotopy.

Since $h_1$ is not nullhomotopic in $S^2 \vee S^2$, this standard nullhomotopy
must have degree $C \neq 0$ (in the sense of relative homology) on at least one
of the $(q+3)$-cells, giving a closed $(q+3)$-form $\omega$ on $Y$ such that
$\int_{S^{q+2} \times I} \omega^*H=L^{2q+2}C$.  Now, suppose $H^\prime$ is some other
nullhomotopy of $h_L$.  Then gluing $H$ and $H^\prime$ along the copies of
$S^{q+2} \times \{0\}$ gives a map $p:S^{q+3} \to Y$.  Since the Hurewicz map sends
$\pi_{q+3}(Y_q)$ to zero, the total degree of $p$ on cells must be zero.  This
shows that $p$ must have degree zero on cells, in other words, that
$\int_{S^{q+2} \times I} \omega^*H^\prime=L^{2q+2}C$.  Thus the volume of a nullhomotopy
of $h_L$ grows at least as $L^{2q+2}$.

In particular, a nullhomotopy $F:S^{q+2} \times [0,1] \to Y_q$ of $h_L$ which has
thickness $\sim L$ has to have width $\sim L^q$.

Now, the rational homotopy groups of $Y_q$ are given by the free Lie algebra on
two generators truncated in degree $q+1$.  A standard computation shows that
differentials of $r$-dimensional generators in the corresponding minimal model
are multiples of the $(r-1)$-dimensional generators.  Therefore, the minimal
depth of the filtration in Conjecture \ref{conjNull} in this case is $q$,
demonstrating that this is in some sense the ``best possible'' conjecture.  In
particular, the bound in Theorem \ref{intro:main} is sharp in at least some
cases.  On the other hand, it is still open whether this quadratic bound is
sharp, for example, for maps $S^3 \to S^2$.

\subsection{Maps that are hard to homotope}
To see that general homotopies do not always behave like nullhomotopies, we
consider maps $S^3 \times S^4 \to S^4$.  Any map $f:S^3 \times S^4 \to S^4$
induces a homomorphism of minimal DGAs
$$\langle x^4, y^7:dx=0,dy=x^2 \rangle \xrightarrow{f^*} \langle a^3, b^4, c^7:
da=db=0, dc=b^2 \rangle$$
which must send $x \mapsto pb$ and $y \mapsto p^2c+qab$ for some $p,q \in
\mathbb{Q}$.  Conversely, for any $p,q \in \mathbb{Z}$ we can define a map
$f_{p,q}$ as illustrated in Figure \ref{fig:s3s4} such that $f_{p,q}^*$ sends $x
\mapsto pb$ and $y \mapsto p^2c+qab$.  This follows from the action of the first
map on cohomology and of the second on homotopy groups.
\begin{figure} 
	\centering
  \begin{tikzpicture}
%    \draw[gray, very thin, step=0.5] (-2,-2) grid (7.6,1);
    \useasboundingbox (-2,-2.5) rectangle (7.6,1);
    \draw (1,1) .. controls (-2,.5) and (-.5,-1.5) .. (1,1);
    \node(S4) at (.1,.4){$S^3$};
    \draw (1,1) .. controls (.8,-2.5) and (4.5,1.2) .. (1,1);
    \node(S4) at (1.8,.3){$S^4$};
    \draw (-.65,-.18) .. controls (0,-.7) and (-.1,-1) .. (-.7,-.8);
    \draw (-.7,-.8) .. controls (-3.5,.4) and (-1.6,-3.5) .. (-.6,-1.5);
    \draw (1.8,-.5) .. controls (.2,-.8) and (0,-.8) .. (-.6,-1.5);
    \draw[->] (2,1) .. controls (2.5,1.2) and (3,1.2) .. (3.7,.8);
    \draw[->] (-.7,-1.8) .. controls (1,-2.1) and (2,-1.9) .. (3.1,-1.6);
    \draw (5,-.5) .. controls (5.2,3) and (1.5,-.5) .. (5,-.5);
    \node(S4) at (4.2,.3){$S^4$};
    \draw (5,-.5) .. controls (4.6,-4) and (1,-.5) .. (5,-.5);
    \node(S4) at (4,-1.3){$S^7$};
    \draw[->] (4.9,.8) .. controls (5.5,1) .. (6.15,.65)
      node[pos=0.5, anchor=south]{degree $p$};
    \draw[->] (4.7,-1.7) .. controls (5.1,-2) .. (6.15,-.65)
      node[pos=0.7, anchor=north west, align=center]{Hopf\\invariant $q$};
    \draw (6.8,0) circle [radius=.8] node{$S^4$};
  \end{tikzpicture}
  
  \caption{
    Construct maps $S^3 \times S^4 \to S^4$ by ``budding off'' a small ball and
    then projecting the rest onto the $S^4$ factor.
  } \label{fig:s3s4}
\end{figure}
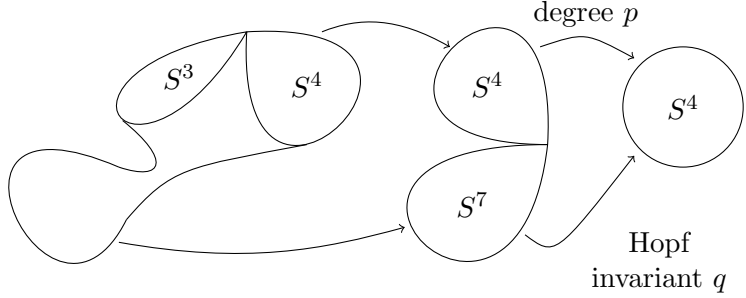

Now, given $p \neq 0$, for any $q_1,q_2 \in \mathbb{Q}$ there is a homotopy of
DGA homomorphisms
$$\langle x^4, y^7:dx=0,dy=x^2 \rangle \to \langle a^3, b^4, c^7:da=db=0, dc=
b^2 \rangle \otimes \langle t^0,dt^1 \rangle$$
between $f^*_{p,q_1}$ and $f^*_{p,q_2}$, given by
\begin{align*}
  x &\mapsto pb+\frac{q_2-q_1}{2p}adt\\
  y &\mapsto p^2c+q_1ab(1-t)+q_2abt.
\end{align*}
This suggests that, at least up to a finite order difference, $f_{p,q_1} \simeq
f_{p,q_2}$.

Indeed, one can see more geometrically that two such maps are homotopic if the
number $\frac{q_2-q_1}{2p}$ is an integer.  For concreteness, suppose $p=1$ and
$q_1=0$.  A potential homotopy between $f_{1,0}$ and $f_{1,q}$ must factor through
the quotient space of $S^3 \times S^4 \times I$ where
$S^3 \times S^4 \times \{0\}$ is projected onto $S^4$ and $S^3 \times S^4 \times
\{1\}$ is mapped onto $S^4 \vee S^7$ as in Figure \ref{fig:s3s4}.  This quotient
space is easily seen to be homeomorphic to $S^4 \times S^4$ minus an open ball,
and thus homotopy equivalent to $S^4 \vee S^4$.  Let $\alpha$ and $\beta$ be the
homotopy classes of the identity maps on the copies of $S^4$, which are images
under the quotient map of $S^3 \times S^4 \times \{0\}$ and $S^3 \times*\times I$
respectively.  Since we know what happens on the ends of the interval, we see
that, if $h:S^4 \vee S^4 \to S^4$ is a map in the homotopy class of such a
homotopy, then $h_*\alpha=[\id_{S^4}]$ and
$$h_*[\alpha,\beta]=\frac{q}{2}[\id_{S^4},\id_{S^4}]=q[\text{Hopf}].$$
Indeed such a map $h$ exists, with $h_*\beta=\frac{q}{2}[\id_{S^4}]$.  In other
words, there is a homotopy $F:S^3 \times S^4 \times I \to S^4$ between $f_{1,0}$
and $f_{1,q}$, and any such homotopy satisfies
$\int_{S^3 \times * \times I} F^*d\vol=\frac{q}{2}$.

If we take $q=L^8$, the way we have defined $f_{1,q}$ gives it Lipschitz constant
$O(L)$.  On the other hand, we have just shown that a homotopy between $f_{1,L^8}$
and $f_{1,0}$ must have degree at least $L^8/2$ on the 4-dimensional submanifold
$S^3 \times * \times I \subset S^3 \times S^4 \times I$.  Thus such a homotopy
must have Lipschitz constant $\Omega(L^2)$, or, if it has linear thickness, it
must have width $\Omega(L^5)$.  Either way, it cannot possibly satisfy the bounds
of Theorem \ref{intro:main}, showing that the theorem cannot directly generalize
beyond nullhomotopies.

\bibliographystyle{amsalpha}
\bibliography{liphom}
\end{document}